\documentclass[11pt,dvips,twoside,letterpaper]{article}

\usepackage{pslatex}
\usepackage{fancyhdr}
\usepackage{graphicx}
\usepackage{geometry}

\RequirePackage{amsfonts,amssymb,amsmath,amscd,amsthm}
\RequirePackage{txfonts}
\RequirePackage{graphicx}
\RequirePackage{xcolor}
\RequirePackage{color}
\RequirePackage{geometry}
\RequirePackage{enumerate}

\usepackage{srcltx}

\def\figurename{Figure} 
\makeatletter
\renewcommand{\fnum@figure}[1]{\figurename~\thefigure.}
\makeatother

\def\tablename{Table} 
\makeatletter
\renewcommand{\fnum@table}[1]{\tablename~\thetable.}
\makeatother

\usepackage{amsmath}
\usepackage{amssymb}
\usepackage{amsfonts}
\usepackage{amsthm,amscd}

\newtheorem{theorem}{Theorem}[section]
\newtheorem{lemma}[theorem]{Lemma}

\theoremstyle{definition}

\theoremstyle{remark}

\numberwithin{equation}{section}

\def\C{\mathbb C}

\def\R{\mathbb R}

\def\Z{\mathbb Z}

\def\N{\mathbb N}

\def\cB{\mathcal{B}}
\def\cE{\mathcal{E}}
\def\cK{\mathcal{K}}
\def\cM{\mathcal{M}}

\def\fA{\mathfrak{A}}
\def\fB{\mathfrak{B}}

\def\fa{\mathfrak{a}}
\def\fc{\mathfrak{c}}
\def\fb{\mathfrak{b}}
\def\fs{\mathfrak{s}}
\def\Op{\operatorname{Op}}
\def\eps{\varepsilon}
\def\Ind{\operatorname{Ind}}

\begin{document}
\title{\bfseries\scshape{ON REGULARIZATION OF MELLIN PDO'S\\ WITH SLOWLY OSCILLATING SYMBOLS\\ OF LIMITED SMOOTHNESS}}
\author{\bfseries\scshape Alexei Yu. Karlovich \thanks{E-mail address: oyk@fct.unl.pt}\\
Centro de Matem\'atica e Aplica\c{c}\~oes (CMA) and Departamento de Matem\'atica, \\
Faculdade de Ci\^encias e Tecnologia, Universidade Nova de Lisboa, \\
Quinta da Torre, 2829--516 Caparica, Portugal\\ \\
\bfseries\scshape Yuri I. Karlovich \thanks{E-mail address: karlovich@uaem.mx}\\
Facultad de Ciencias,
Universidad Aut\'onoma del Estado de Morelos,\\
Av. Universidad 1001, Col. Chamilpa,\\
C.P. 62209 Cuernavaca, Morelos, M\'exico\\ \\
\bfseries\scshape Amarino B. Lebre\thanks{E-mail address: alebre@math.tecnico.ulisboa.pt}\\
Departamento de Matem\'atica, Instituto Superior T\'ecnico,\\
Universidade de Lisboa, Av. Rovisco Pais, 1049--001 Lisboa, Portugal
\\ \\ \\ }

\date{}
\maketitle \thispagestyle{empty} \setcounter{page}{1}


\begin{abstract} \noindent
We study Mellin pseudodifferential operators (shortly, Mellin PDO's) with symbols
in the algebra $\widetilde{\cE}(\R_+,V(\R))$ of slowly oscillating functions of
limited  smoothness introduced in \cite{K09}. We show that if
$\fa\in\widetilde{\cE}(\R_+,V(\R))$ does not degenerate on the ``boundary" of
$\R_+\times\R$ in a certain sense,  then the Mellin PDO $\Op(\fa)$ is Fredholm
on the space $L^p$ for $p\in(1,\infty)$ and each its regularizer is of the form
$\Op(\fb)+K$ where $K$ is a compact operator on $L^p$ and $\fb$ is a certain
explicitly  constructed function in the same algebra $\widetilde{\cE}(\R_+,V(\R))$
such that  $\fb=1/\fa$ on the ``boundary" of $\R_+\times\R$. This result complements
a known Fredholm criterion from \cite{K09} for Mellin PDO's with symbols in the
closure of $\widetilde{\cE}(\R_+,V(\R))$.
\end{abstract}

\noindent {\bf AMS Subject Classification:} Primary 47G30; Secondary 47A53.

\vspace{.08in} \noindent \textbf{Keywords}:
Fredholmness,
regularizer,
Mellin pseudodifferential operator,
slowly oscillating symbol,
maximal ideal space.

\newpage

\pagestyle{fancy} \fancyhead{}
\fancyhead[EC]{A. Karlovich, Yu. Karlovich, and A. Lebre}
\fancyhead[EL,OR]{\thepage}
\fancyhead[OC]{On Regularization of Mellin PDO's}
\fancyfoot{}
\renewcommand\headrulewidth{0.5pt}

\section{Introduction}

Let $\cB(X)$ be the Banach algebra of all bounded linear operators
acting on a Banach space $X$, and let $\cK(X)$ be the ideal of all
compact operators in $\cB(X)$. An operator $A\in\cB(X)$ is called
\textit{Fredholm} if its image is closed and the spaces $\ker A$
and $\ker A^*$ are finite-dimensional.
 In that case the number
\[
\Ind A:=\dim\ker A-\dim\ker A^*
\]
is referred to as the {\it index} of $A$ (see, e.g.,
\cite[Sections~1.11--1.12]{BS06}, \cite[Chap.~4]{GK92}).
For bounded linear operators $A$ and $B$, we will write $A\simeq B$ if
$A-B\in\cK(X)$.

Recall that an  operator $B_r\in\cB(X)$ (resp. $B_l\in\cB(X)$)
is said to be a right  (resp. left) regularizer for $A$ if
\[
AB_r\simeq I \quad(\mbox{resp.}\quad B_lA\simeq I).
\]
It is well known that the operator $A$ is Fredholm on $X$ if and only if it
admits simultaneously a right and a left regularizers. Moreover, each right
regularizer differs from each left regularizer by a compact operator
(see, e.g., \cite[Chap.~4, Section 7]{GK92}). Therefore we may speak of a
regularizer $B=B_r=B_l$ of $A$ and two different regularizers of $A$ differ
from each other by a compact operator.

Let $d\mu(t)=dt/t$ be the (normalized) invariant measure on $\R_+$.
Consider the Fourier transform on $L^2(\mathbb{R}_+,d\mu)$, which is
usually referred to as the Mellin transform and is defined by
\[
\cM:L^2(\R_+,d\mu)\to L^2(\R),
\quad
(\cM f)(x):=\int_{\R_+} f(t) t^{-ix}\,\frac{dt}{t}.
\]
It is an invertible operator, with inverse given by
\[
{\cM^{-1}}:L^2(\R)\to L^2(\R_{+},d\mu),
\quad
({\cM^{-1}}g)(t)= \frac{1}{2\pi}\int_{\R}
g(x)t^{ix}\,dx.
\]
For $1<p<\infty$, let $\cM_p$ denote the Banach algebra of all Mellin multipliers,
that is, the set of all functions $a\in L^\infty(\R)$ such that $\cM^{-1}a\cM f\in L^p(\R_+,d\mu)$
and
\[
\|\cM^{-1}a\cM f\|_{L^p(\R_+,d\mu)}\le c_p\|f\|_{L^p(\R_+,d\mu)}
\quad\mbox{for all}\quad
f\in L^2(\R_+,d\mu)\cap L^p(\R_+,d\mu).
\]
If $a\in\cM_p$, then the operator
$f\mapsto \cM^{-1}a\cM f$ defined initially on $L^2(\R_+,d\mu)\cap L^p(\R_+,d\mu)$
extends to a bounded operator on $L^p(\R_+,d\mu)$. This operator is called the Mellin
convolution operator with symbol $a$.

Mellin pseudodifferential operators are generalizations of Mellin convolution
operators. Let $\fa$ be a sufficiently smooth function defined on $\R_+\times\R$.
The Mellin pseudodifferential operator (shortly, Mellin PDO) with symbol $\fa$
is initially defined for smooth functions $f$ of compact support by the iterated
integral
\[
[\Op(\fa) f](t)=
[\cM^{-1}\fa(t,\cdot)\cM f](t)=
\frac{1}{2\pi}\int_\R dx \int_{\R_+}
\fa(t,x)\left(\frac{t}{\tau}\right)^{ix}f(\tau) \frac{d\tau}{\tau}
\quad\mbox{for}\quad t\in\R_+.
\]

In 1991 Rabinovich \cite{R92} proposed to use Mellin pseudodifferential
operators  techniques to study singular integral operators
on slowly oscillating Carleson curves. This idea was exploited in a series of
papers by Rabinovich and coauthors (see, e.g., \cite{R95,R98} and
\cite[Sections~4.5--4.6]{RRS04} and the references therein). Rabinovich
stated in \cite[Theorem~2.6]{R98} a Fredholm criterion for Mellin PDO's
with $C^\infty$ slowly oscillating (or slowly varying) symbols
on the spaces $L^p(\R_+,d\mu)$ for $1<p<\infty$. Namely,
he considered symbols $\fa\in C^\infty(\R_+\times\R)$ such that
\begin{equation}\label{eq:Hoermander}
\sup_{(t,x)\in\R_+\times\R}
\big|(t\partial_t)^j\partial_x^k\fa(t,x)\big|(1+x^2)^{k/2}<\infty
\quad\mbox{for all}\quad j,k\in\Z_+
\end{equation}
and
\begin{equation}\label{eq:Grushin}
\lim_{t\to s}\sup_{x\in\R}
\big|(t\partial_t)^j\partial_x^k\fa(t,x)\big|(1+x^2)^{k/2}=0
\quad\mbox{for all}\quad j\in\N,\quad k\in\Z_+, \quad s\in\{0,\infty\}.
\end{equation}
Here and in what follows $\partial_t$ and $\partial_x$ denote the operators of
partial  differentiation with respect to $t$ and to $x$. Notice that
\eqref{eq:Hoermander}  defines nothing but the Mellin version of the H\"ormander
class $S_{1,0}^0(\R)$ (see, e.g., \cite{H67}, \cite[Chap.~2, Section~1]{K82} for
the definition of the  H\"ormander classes $S_{\varrho,\delta}^m(\R^n)$). If
$\fa$ satisfies  \eqref{eq:Hoermander}, then the Mellin PDO $\Op(\fa)$ is bounded
on the spaces $L^p(\R_+,d\mu)$ for $1<p<\infty$ (see, e.g.,
\cite[Chap.~VI, Proposition~4]{St93} for the corresponding Fourier PDO's).
Condition \eqref{eq:Grushin}
is the Mellin version of Grushin's definition of slowly varying symbols in the
first variable (see, e.g., \cite{G70}, \cite[Chap.~3, Defintion~5.11]{K82}).

The above mentioned results have a disadvantage that the smoothness conditions
imposed on slowly oscillating symbols are very strong. In this paper we will
use a much weaker notion of slow oscillation, which goes back to Sarason \cite{S77}.
A bounded continuous function $f$ on $\R_+=(0,\infty)$ is called slowly oscillating
at $0$ and $\infty$ if
\[
\lim_{r\to s}\max_{t,\tau\in[r,2r]}|f(t)-f(\tau)|=0
\quad\mbox{for}\quad
s\in\{0,\infty\}.
\]
This definition can be extended to the case of bounded continuous
functions on $\R_+$ with values in a Banach space $X$.

The set $SO(\R_+)$ of all slowly oscillating functions forms a $C^*$-algebra.
This algebra properly contains $C(\overline{\R}_+)$, the $C^*$-algebra of all
continuous functions on $\overline{\R}_+:=[0,+\infty]$. For a unital commutative
Banach algebra $\fA$, let $M(\mathfrak{A})$ denote its maximal ideal space.
Identifying the points $t\in\overline{\R}_+$ with the evaluation functionals
$t(f)=f(t)$ for $f\in C(\overline{\R}_+)$, we get
$M(C(\overline{\R}_+))=\overline{\R}_+$. Consider the fibers
\[
M_s(SO(\R_+))
:=
\big\{\xi\in M(SO(\R_+)):\xi|_{C(\overline{\R}_+)}=s\big\}
\]
of the maximal ideal space $M(SO(\R_+))$ over the points
$s\in\{0,\infty\}$. By \cite[Proposition~2.1]{K09}, the set
\[
\Delta:=M_0(SO(\R_+))\cup M_\infty(SO(\R_+))
\]
coincides with $(\operatorname{clos}_{SO^*}\R_+)\setminus\R_+$
where $\operatorname{clos}_{SO^*}\R_+$ is the weak-star closure
of $\R_+$ in the dual space of $SO(\R_+)$. Then $M(SO(\R_+))=\Delta\cup\R_+$.

The second author \cite{K06} developed a Fredholm theory for Fourier
pseudodifferential operators with slowly oscillating $V(\R)$-valued symbols
where $V(\R)$ is the Banach algebra of absolutely continuous functions
of bounded total variation on $\R$. Those results were translated to the
Mellin setting in \cite{K09}. In particular, the important algebra
$\widetilde{\cE}(\R_+,V(\R))$ of slowly oscillating $V(\R)$-valued functions
was introduced and a Fredholm criterion for Mellin PDO's with symbols in the
closure of $\widetilde{\cE}(\R_+,V(\R))$ in the norm of the Banach algebra
$C_b(\R_+,C_p(\R))$ of bounded continuous $C_p(\R)$-valued functions was
obtained on the space $L^p(\R,d\mu)$ for all $p\in(1,\infty)$ \cite[Theorem~4.3]{K09}.
Here $C_p(\R)$  is the smallest closed subalgebra of the algebra $\cM_p(\R)$ that
contains the algebra $V(\R)$. We refer, e.g., to \cite[Sections~9.1--9.7]{BS06},
\cite[Chap.~1]{D79}, \cite[Section~2.1]{HRS94}, \cite[Section~4.2]{RSS11},
and \cite{SM86} for properties of the algebras $V(\R)$, $C_p(\R)$, and $\cM_p(\R)$.

For symbols in the algebra $\widetilde{\cE}(\R_+,V(\R))$ the above mentioned
Fredholm criterion has a simpler form \cite[Theorem~3.6]{KKL14}. That result
was already used in \cite{K14} (see also \cite{KKL14}) to prove that the
simplest weighted singular integral operator with two shifts
\begin{equation}\label{eq:simplest-SIOS}
U_\alpha P_\gamma^++U_\beta P_\gamma^-
\end{equation}
is Fredholm of index zero on the space $L^p(\R_+)$ with $p\in(1,\infty)$,
where $\alpha,\beta:\R_+\to\R_+$ are
orientation preserving diffeomorphisms with the only fixed points $0$ and $\infty$
such that $\log\alpha',\log\beta'$ are bounded, $\alpha',\beta'\in SO(\R_+)$,
\[
U_\alpha f=(\alpha')^{1/p}(f\circ\alpha),
\quad
U_\beta f=(\beta')^{1/p}(f\circ\beta),
\quad
P_\gamma^\pm:=(I\pm S_\gamma)/2,
\]
and $S_\gamma$ is the weighted Cauchy singular integral operator given by
\[
(S_\gamma f)(t):=\frac{1}{\pi i}
\int_{\R_+}\left(\frac{t}{\tau}\right)^\gamma\frac{f(\tau)}{\tau-t}d\tau
\]
with $\gamma\in\C$ satisfying $0<1/p+\Re\gamma<1$ (for $\gamma=0$ this result
was obtained in \cite{KKL14}). To study more general operators than
\eqref{eq:simplest-SIOS} in the forthcoming paper \cite{KKL-preparation},
we need not only a Fredholm criterion for $\Op(\fa)$ with
$\fa\in\widetilde{\cE}(\R_+,V(\R))$ given in \cite[Theorem~3.6]{KKL14},
but also an information on the regularizers of $\Op(\fa)$. Note that a
full description of the regularizers of a Fredholm Mellin PDO $\Op(\fa)$
is available if $\fa\in C^\infty(\R_+\times\R)$ satisfies
\eqref{eq:Hoermander}--\eqref{eq:Grushin}, see \cite[Theorem~2.6]{R98}),
however such a description is missing for the algebra $\widetilde{\cE}(\R_+,V(\R))$.

The aim of this paper is to fill in this gap and to complement the Fredholm
criterion for Mellin PDO's with symbols in $\widetilde{\cE}(\R_+,V(\R))$. Here
we provide an explicit description of all regularizers of a Fredholm operator
$\Op(\fa)$ with $\fa\in\widetilde{\cE}(\R_+,V(\R))$. Namely, we prove that if
$\fa\in\widetilde{\cE}(\R_+,V(\R))$ does not degenerate on the ``boundary" of
$\R_+\times\R$ in a certain sense, then the Mellin PDO $\Op(\fa)$ is Fredholm
on the space $L^p(\R_+,d\mu)$ for $p\in(1,\infty)$ and each its regularizer is
of the form $\Op(\fb)+K$ where $K$ is a compact operator on $L^p(\R_+,d\mu)$
and $\fb$ is a certain explicitly  constructed function in the same algebra
$\widetilde{\cE}(\R_+,V(\R))$ such that  $\fb=1/\fa$ on the ``boundary" of
$\R_+\times\R$. By the ``boundary" of $\R_+\times\R$ we mean the set
\begin{equation}\label{eq:boundary}
(\R_+\times\{\pm\infty\})\cup(\Delta\times\overline{\R}).
\end{equation}

The paper is organized as follows. In Section~\ref{sec:Boundedness}
we define the algebra $C_b(\R_+,V(\R))$ of all bounded continuous
$V(\R)$-valued functions and state that if $\fa\in C_b(\R_+,V(\R))$, then
$\Op(\fa)$ is bounded on $L^p(\R_+,d\mu)$. In Section~\ref{sec:Compactness}
we introduce the algebra $SO(\R_+,V(\R))$ of slowly oscillating $V(\R)$-valued
functions (a generalization of $SO(\R_+))$ and its subalgebra $\cE(\R_+,V(\R))$.
Further we explain how the values of a function $\fa\in\cE(\R_+,V(\R))$ on the
boundary \eqref{eq:boundary} are defined and recall that
\begin{equation}\label{eq:compactness-semi-commutators}
\Op(\fa)\Op(\fb)\simeq \Op(\fa\fb)
\quad\mbox{whenever}\quad
\fa,\fb\in \cE(\R_+,V(\R)).
\end{equation}
In Section~\ref{sec:Inverse-closedness} we define our main algebra
$\widetilde{\cE}(\R_+,V(\R))\subset\cE(\R_+,V(\R))$ and show that all
algebras $C_b(\R_+,V(\R))$, $SO(\R_+,V(\R))$, $\cE(\R_+,V(\R))$, and
$\widetilde{\cE}(\R_+,V(\R))$ are inverse closed in $C_b(\R_+\times\R)$,
the algebra of all bounded continuous functions on $\R_+\times\R$. Combining
the inverse closedness of the algebras $\cE(\R_+,V(\R))$ (resp.
$\widetilde{\cE}(\R_+,V(\R))$) with \eqref{eq:compactness-semi-commutators},
we get a description of all regularizers for $\Op(\fa)$ with
$\fa\in\cE(\R_+,V(\R))$ (resp. $\widetilde{\cE}(\R_+,V(\R))$)
bounded away from zero on $\R_+\times\R$. In Section~\ref{sec:Fredholmness}
we show that the latter strong hypothesis can be essentially relaxed
in the case of the algebra $\widetilde{\cE}(\R_+,V(\R))$. We show that if
$\fa\in\widetilde{\cE}(\R_+,V(\R))$ does not degenerate on the ``boundary"
\eqref{eq:boundary}, then there exists $\fb\in\widetilde{\cE}(\R_+,V(\R))$
such that $\fb=1/\fa$ on the ``boundary" \eqref{eq:boundary}. This construction
becomes possible for $\fa\in\widetilde{\cE}(\R_+,V(\R))$ because the limiting
values of $\fa(t,\cdot)$ on $\Delta$ are attained uniformly in the norm of $V(\R)$
(see Lemma~\ref{le:values-tilde}). Finally we recall that if
$\fc\in\widetilde{\cE}(\R_+,V(\R))$, then $\Op(\fc)$ is compact if and only
if its symbol $\fc$ degenerates on the  ``boundary" \eqref{eq:boundary}.
Combining this result with our construction, we arrive at the main result
of the paper.
\section{Algebra \boldmath{$C_b(\R_+,V(\R))$} and Boundedness of Mellin PDO's}
\label{sec:Boundedness}
\subsection{Definition of the Algebra  \boldmath{$C_b(\R_+,V(\R))$}}
Let $a$ be an absolutely continuous function of finite total variation
\[
V(a):=\int_\R|a'(x)|dx
\]
on $\R$. The set $V(\R)$ of all absolutely continuous functions of finite
total variation on $\R$ becomes a Banach algebra equipped with the norm
\begin{equation}\label{eq:norm-V}
\|a\|_V:=\|a\|_{L^\infty(\R)}+V(a).
\end{equation}
Following \cite{K06,K06-IWOTA}, let $C_b(\R_+,V(\R))$ denote the Banach
algebra of all bounded continuous $V(\R)$-valued functions on $\R_+$ with
the norm
\[
\|\fa(\cdot,\cdot)\|_{C_b(\R_+,V(\R))}
=
\sup_{t\in\R_+}\|\fa(t,\cdot)\|_V.
\]
\subsection{Boundedness of Mellin PDO's}
As usual, let $C_0^\infty(\R_+)$ be the set of all infinitely differentiable
functions of compact support on $\R_+$.

The following boundedness result for Mellin pseudodifferential operators
can be extracted from \cite[Theorem~6.1]{K06-IWOTA} (see also \cite[Theorem~3.1]{K06}).
\begin{theorem}\label{th:boundedness-PDO}
If $\fa\in C_b(\R_+,V(\R))$, then the Mellin pseudodifferential operator
$\operatorname{Op}(\fa)$, defined for functions $f\in C_0^\infty(\R_+)$ by
the iterated integral
\[
[\operatorname{Op}(\fa)f](t)
=
\frac{1}{2\pi}\int_\R dx \int_{\R_+}
\fa(t,x)\left(\frac{t}{\tau}\right)^{ix}f(\tau) \frac{d\tau}{\tau}
\quad\mbox{for}\quad t\in\R_+,
\]
extends to a bounded linear operator on the space $L^p(\R_+,d\mu)$
and there is a positive constant $C_p$ depending only on $p$ such
that
\[
\|\operatorname{Op}(\fa)\|_{\cB(L^p(\R_+,d\mu))}
\le C_p\|\fa\|_{C_b(\R_+,V(\R))}.
\]
\end{theorem}
\section{Algebra \boldmath{$\cE(\R_+,V(\R))$} and Compactness of Semi-Commutators of Mellin PDO's}
\label{sec:Compactness}
\subsection{Definitions of the Algebras \boldmath{$SO(\R_+,V(\R))$} and \boldmath{$\cE(\R_+,V(\R))$}}
Let $SO(\R_+,V(\R))$ denote the Banach subalgebra of $C_b(\R_+,V(\R))$
consisting of all $V(\R)$-valued functions $\fa$ on $\R_+$ that slowly
oscillate at $0$ and $\infty$, that is,
\[
\lim_{r\to 0} \operatorname{cm}_r^C(\fa)
=
\lim_{r\to \infty} \operatorname{cm}_r^C(\fa)=0,
\]
where
\[
\operatorname{cm}_r^C(\fa)
:=
\max\big\{
\|\fa(t,\cdot)-\fa(\tau,\cdot)\|_{L^\infty(\R)}:t,\tau\in[r,2r]
\big\}.
\]

Let $\cE(\R_+,V(\R))$ be the Banach algebra of all $V(\R)$-valued functions
$\fa\in SO(\R_+,V(\R))$ such that
\begin{equation}\label{eq:definition-cE}
\lim_{|h|\to 0}\sup_{t\in\R_+}\|\fa(t,\cdot)-\fa^h(t,\cdot)\|_V=0
\end{equation}
where $\fa^h(t,x):=\fa(t,x+h)$ for all $(t,x)\in\R_+\times \R$.
\subsection{Limiting Values of Functions in the Algebra \boldmath{$\cE(\R_+,V(\R))$}}
\label{sec:values}
Let $\fa\in\cE(\R_+,V(\R))$. For every $t\in\R_+$, the function $\fa(t,\cdot)$
belongs to $V(\R)$ and, therefore, has finite limits at $\pm\infty$, which will
be denoted by $\fa(t,\pm\infty)$. Now we explain how to extend the function
$\fa$ to $\Delta\times\overline{\R}$. By analogy with \cite[Lemma~2.7]{K06}
one can prove the following.
\begin{lemma}\label{le:values}
Let $s\in\{0,\infty\}$ and $\{\fa_k\}_{k=1}^\infty$ be a countable subset of the
algebra $\cE(\R_+,V(\R))$. For each $\xi\in M_s(SO(\R_+))$ there is a sequence
$\{t_j\}_{j\in\N}\subset\R_+$ and functions
$\fa_k(\xi,\cdot)\in V(\R)$ such that $t_j\to s$ as $j\to\infty$ and
\[
\fa_k(\xi,x)=\lim_{j\to\infty}\fa_k(t_j,x)
\]
for every $x\in\overline{\R}$ and every $k\in\N$.
\end{lemma}
The following lemma will be of some importance in applications we have in mind
\cite{KKL-preparation} (although it will not be used in the current paper).
\begin{lemma}\label{le:values-series}
Let $\{\fa_n\}_{n\in\mathbb{N}}$ be a sequence of functions in $\cE(\R_+,V(\R))$
such that the series $\sum_{n=1}^\infty\fa_n$ converges in the norm of $C_b(\R_+,V(\R))$
to a function $\fa\in\cE(\R_+,V(\R))$. Then
\begin{equation}\label{eq:values-series-1}
\fa(t,\pm\infty) =\sum_{n=1}^\infty \fa_n(t,\pm\infty)
\ \mbox{ for all }\
t\in\R_+,
\quad
\fa(\xi,x) =\sum_{n=1}^\infty \fa_n(\xi,x)
\ \mbox{ for all }\
(\xi,x)\in\Delta\times\R.
\end{equation}
\end{lemma}
\begin{proof}
Fix  $\eps>0$. For $N\in\N$, put
\[
\fs_N:=\sum_{n=1}^N \fa_n.
\]
By the hypothesis, there exists $N_0\in\N$ such that for all $N>N_0$,
\begin{equation}\label{eq:values-series-3}
\sup_{(t,x)\in\R_+\times\R}|\fa(t,x)-\fs_N(t,x)|
\le
\|\fa-\fs_N\|_{C_b(\R_+,V(\R))}
<\eps/3.
\end{equation}
Fix some $t\in\R_+$. For every $N>N_0$ there exists $x(t,N)\in\R_+$ such that
for all $x\in (x(t,N),+\infty)$,
\begin{equation}\label{eq:values-series-4}
|\fa(t,+\infty)-\fa(t,x)|<\eps/3,
\quad
|\fs_N(t,+\infty)-\fs_N(t,x)|<\eps/3.
\end{equation}
From \eqref{eq:values-series-3} and \eqref{eq:values-series-4} it follows that
for every $N>N_0$ and $x\in (x(t,N),+\infty)$,
\[
|\fa(t,+\infty)-\fs_N(t,+\infty)|
\le
|\fa(t,+\infty)-\fa(t,x)|
+
|\fa(t,x)-\fs_N(t,x)|
+
|\fs_N(t,x)-\fs_N(t,+\infty)|
<\eps.
\]
This implies the first equality in \eqref{eq:values-series-1} for the sign ``$+$".
The proof for the sign ``$-$" is analogous.

Fix $s\in\{0,\infty\}$ and $\xi\in M_s(SO(\R_+))$. In view of Lemma~\ref{le:values},
there exists a sequence $\{t_j\}_{j\in\N}\subset\R_+$ such that
$t_j\to s$ as $j\to\infty$ and functions $\fa(\xi,\cdot)\in V(\R_+)$ and
$\fs_N(\xi,\cdot)\in V(\R_+)$, $N\in\N$, such that
\[
\fa(\xi,x)=\lim_{j\to\infty}\fa(t_j,x),
\quad
\fs_N(\xi,x)=\lim_{j\to\infty}\fs_N(t_j,x)
\]
for all $x\in\overline{\R}$ and all $N\in\N$.

Fix $x\in\R$. For every $N>N_0$ there exists $j_0(x,N)\in\N$ such that
for $j>j_0(x,N)$,
\begin{equation}\label{eq:values-series-5}
|\fa(\xi,x)-\fa(t_j,x)|<\eps/3,
\quad
|\fs_N(\xi,x)-\fs_N(t_j,x)|<\eps/3.
\end{equation}
From \eqref{eq:values-series-3} and \eqref{eq:values-series-5}
we obtain that for $N>N_0$ and $j>j_0(x,N)$,
\[
\left|\fa(\xi,x)-\fs_N(\xi,x)\right|
\le
|\fa(\xi,x)-\fa(t_j,x)|
+
|\fa(t_j,x)-\fs_N(t_j,x)|+|\fs_N(t_j,x)-\fs_N(\xi,x)|
<
\eps,
\]
which concludes the proof of the second equality in \eqref{eq:values-series-1}.
\end{proof}
\subsection{Compactness of Semi-Commutators of Mellin PDO's}
Let $E$ be the isometric isomorphism
\begin{equation}\label{eq:def-E}
E:L^p(\R_+,d\mu)\to L^p(\R),
\quad
(Ef)(x):=f(e^x),\quad x\in\R.
\end{equation}
Applying the relation
\begin{equation}\label{eq:translation-PDO}
\Op(\fa)=E^{-1}a(x,D)E
\end{equation}
between the Mellin pseudodifferential operator $\Op(\fa)$ and the Fourier
pseudodifferential operator $a(x,D)$ considered in \cite{K06}, where
\begin{equation}\label{eq:translation-symbols}
\fa(t,x)=a(\ln t,x),\quad(t,x)\in\R_+\times\R,
\end{equation}
we infer from \cite[Theorem~8.3]{K06} the following compactness result.
\begin{theorem}\label{th:comp-semi-commutators-PDO}
If $\fa,\fb\in\cE(\R_+,V(\R))$, then
$\operatorname{Op}(\fa)\operatorname{Op}(\fb)\simeq \operatorname{Op}(\fa\fb).$
\end{theorem}
\section{Regularization of Mellin PDO's with Symbols\\ Globally Bounded Away from Zero}
\label{sec:Inverse-closedness}
\subsection{Definition of the Algebra \boldmath{$\widetilde{\cE}(\R_+,V(\R))$}}
We denote by $\widetilde{\cE}(\R_+,V(\R))$ the Banach algebra
consisting of all functions $\fa\in\cE(\R_+,V(\R))$ that satisfy the condition
\begin{equation}\label{eq:definition-tilde-cE}
\lim_{m\to\infty}\sup_{t\in\R_+}\int_{\R\setminus[-m,m]}
|\partial_x\fa(t,x)|\,dx=0.
\end{equation}
This algebra plays a crucial role in the paper.

\subsection{Inverse Closedness of the Algebras \boldmath{$C_b(\R_+,V(\R)$},
\boldmath{$SO(\R_+,V(\R))$},\\ \boldmath{$\cE(\R_+,V(\R))$}, and
\boldmath{$\widetilde{\cE}(\R_+,V(\R))$} in the Algebra \boldmath{$C_b(\R_+\times\R)$}}
Let $\fB$ be a unital Banach algebra and $\fA$ be a subalgebra of $\fB$, which
contains the identity element of $\fB$. The algebra $\fA$ is said to be inverse
closed in the algebra $\fB$ if every element $a\in\fA$, invertible in $\fB$,
is invertible in $\fA$ as well.
\begin{lemma}\label{le:inverse-closedness}
The algebras $C_b(\R_+,V(\R))$, $SO(\R_+,V(\R))$, $\cE(\R_+,V(\R))$, and
$\widetilde{\cE}(\R_+,V(\R))$ are inverse  closed in the Banach algebra
$C_b(\R_+\times\R)$ of all bounded continuous functions on the half-plane
$\R_+\times\R$.
\end{lemma}
\begin{proof}
The proof is developed by analogy with \cite[pp.~755--756]{K06}.
Let $\fa\in C_b(\R_+,V(\R))$ be invertible in $C_b(\R_+\times\R)$. Then
\[
\|\fa^{-1}\|_{C_b(\R_+\times\R)}
=
\sup_{(t,x)\in\R_+\times\R}|\fa^{-1}(t,x)|
=
\left(\inf_{(t,x)\in\R_+\times\R}|\fa(t,x)|\right)^{-1}<\infty.
\]
Therefore, for every $t\in\R_+$,
\begin{align}
\|\fa^{-1}(t,\cdot)\|_V
&=
\|\fa^{-1}(t,\cdot)\|_{L^\infty(\R_+)}+V(\fa^{-1}(t,\cdot))
=
\sup_{x\in\R}\left|\frac{\fa(t,x)}{\fa^2(t,x)}\right|
+
\int_\R\left|\frac{\partial_x\fa(t,x)}{\fa^2(t,x)}\right|dx
\nonumber\\
&\le
\|\fa^{-1}\|_{C_b(\R_+\times\R)}^2
\left(
\|\fa(t,\cdot)\|_{L^\infty(\R)}
+
V(\fa(t,\cdot))
\right)
=
\|\fa^{-1}\|_{C_b(\R_+\times\R)}^2
\|\fa(t,\cdot)\|_V.
\label{eq:IC-1}
\end{align}
Hence
\begin{equation}\label{eq:IC-2}
\|\fa^{-1}(\cdot,\cdot)\|_{C_b(\R_+,V(\R))}
\le
\|\fa^{-1}\|_{C_b(\R_+\times\R)}^2
\|\fa(\cdot,\cdot)\|_{C_b(\R_+,V(\R))}
\end{equation}
and for every $t,\tau\in\R_+$,
\begin{align}
\|\fa^{-1}(t,\cdot)-\fa^{-1}(\tau,\cdot)\|_V
&\le
\|\fa^{-1}(t,\cdot)\|_V\|\fa^{-1}(\tau,\cdot)\|_V
\|\fa(t,\cdot)-\fa(\tau,\cdot)\|_V
\notag\\
&\le
\|\fa^{-1}\|_{C_b(\R_+\times\R)}^4\|\fa(\cdot,\cdot)\|_{C_b(\R_+,V(\R))}
\|\fa(t,\cdot)-\fa(\tau,\cdot)\|_V.
\label{eq:IC-3}
\end{align}
From inequalities \eqref{eq:IC-2}--\eqref{eq:IC-3} it follows that the function
$\fa^{-1}$ is a bounded and continuous $V(\R)$-valued function.
Thus, $C_b(\R_+,V(\R))$ is inverse closed in $C_b(\R_+\times\R)$.

Suppose $\fa\in SO(\R_+,V(\R))$ is invertible in $C_b(\R_+\times\R)$. If
$t,\tau\in\R_+$, then
\begin{equation}\label{eq:IC-4}
\|\fa^{-1}(t,\cdot)-\fa^{-1}(\tau,\cdot)\|_{L^\infty(\R)}
\le
\|\fa^{-1}\|_{C_b(\R_+\times\R)}^2
\|\fa(t,\cdot)-\fa(\tau,\cdot)\|_{L^\infty(\R)}.
\end{equation}
Therefore
\[
\operatorname{cm}_r^C(\fa^{-1})
\le
\|\fa^{-1}\|_{C_b(\R_+\times\R)}^2\operatorname{cm}_r^C(\fa),
\quad r\in\R_+.
\]
From the above inequality we conclude that $\fa^{-1}\in SO(\R_+,V(\R))$.
Thus, $SO(\R_+,V(\R))$ is inverse closed in $C_b(\R_+\times\R)$.

Let $\fa\in \cE(\R_+,V(\R))$ be invertible in $C_b(\R_+\times\R)$.
Taking into account inequality \eqref{eq:IC-1} and that the norm in $V(\R)$
is translation-invariant, we get for $h\in\R$ and $t\in\R_+$,
\begin{align}
\|\fa^{-1}(t,\cdot) -(\fa^{-1})^h(t,\cdot)\|_V
&\le
\|\fa^{-1}(t,\cdot)\|_V
\|(\fa^{-1})^h(t,\cdot)\|_V
\|\fa(t,\cdot)-(\fa)^h(t,\cdot)\|_V
\notag\\
&\le
\|\fa^{-1}\|_{C_b(\R_+\times\R)}^4
\|\fa(\cdot,\cdot)\|_{C_b(\R_+,V(\R))}^2
\|\fa(t,\cdot)-\fa^h(t,\cdot)\|_V.
\label{eq:IC-5}
\end{align}
From the above inequality and $\fa\in\cE(\R_+,V(\R))$ it follows that
\begin{align*}
\lim_{|h|\to 0}\sup_{t\in\R_+}
\|\fa^{-1}(t,\cdot)-(\fa^{-1})^h(t,\cdot)\|_V=0.
\end{align*}
This means that $\fa^{-1}\in\cE(\R_+,V(\R))$, whence the proof of the
inverse closedness of the algebra $\cE(\R_+,V(\R))$ in the algebra
$C_b(\R_+\times\R)$ is completed.

Finally, if $\fa\in\widetilde{\cE}(\R_+,V(\R))$ is invertible in
$C_b(\R_+\times\R)$, then
\[
\lim_{m\to\infty}\sup_{t\in\R_+}\int_{\R\setminus[-m,m]}
|\partial_x\fa^{-1}(t,x)|\,dx
\le
\|\fa^{-1}\|_{C_b(\R_+\times\R)}^2
\lim_{m\to\infty}\sup_{t\in\R_+}\int_{\R\setminus[-m,m]}
|\partial_x\fa(t,x)|\,dx
=0.
\]
Therefore, $\fa^{-1}\in\widetilde{\cE}(\R_+,V(\R))$ and thus the algebra
$\widetilde{\cE}(\R_+,V(\R))$ is inverse closed in the algebra $C_b(\R_+,V(\R))$.
\end{proof}
\subsection{First Result on the Regularization of Mellin PDO's}
\begin{lemma}
If $\fa\in\cE(\R_+,V(\R))$ (resp. $\fa\in\widetilde{\cE}(\R_+,V(\R))$ is such that
\begin{equation}\label{eq:strong-assumption}
\inf_{(t,x)\in\R_+\times\R}|\fa(t,x)|>0,
\end{equation}
then the Mellin pseudodifferential operator $\Op(\fa)$ is Fredholm on the space
$L^p(\R_+,d\mu)$ and each its regularizer is of the form $\Op(1/\fa)+K$ where
$K$ is a compact operator on the space $L^p(\R_+,d\mu)$ and $1/\fa\in\cE(\R_+,V(\R))$
(resp. $1/\fa\in\widetilde{\cE}(\R_+,V(\R))$).
\end{lemma}
\begin{proof}
If $\fa$ satisfies \eqref{eq:strong-assumption} and belongs to $\cE(\R_+,V(\R))$
(resp. to $\widetilde{\cE}(\R_+,V(\R))$), then $1/\fa$ belongs to $\cE(\R_+,V(\R))$
(resp. to $\widetilde{\cE}(\R_+,V(\R))$) in view of Lemma~\ref{le:inverse-closedness}.
Then in both cases from Theorem~\ref{th:comp-semi-commutators-PDO} we obtain
$\Op(\fa)\Op(1/\fa)\simeq \Op(1)=I$ and $\Op(1/\fa)\Op(\fa)\simeq\Op(1)=I$, which
completes the proof.
\end{proof}
As it happens, the very strong hypothesis \eqref{eq:strong-assumption} can be
essentially relaxed for Mellin PDO's with symbols in the algebra $\widetilde{\cE}(\R_+,V(\R))$.
This issue will be discussed in the next section.
\section{Algebra \boldmath{$\widetilde{\cE}(\R_+,V(\R))$} and Fredholmness of Mellin PDO's}
\label{sec:Fredholmness}
\subsection{Elementary Properties of Two Important Functions in \boldmath{$V(\R)$}}
We prelude our main construction with properties of two important functions in $V(\R)$.
\begin{lemma}\label{le:elementary}
\begin{enumerate}
\item[{\rm(a)}]
For $x\in\R$, put
\begin{equation}\label{eq:elementary-1}
p_-(x):=(1-\tanh(\pi x))/2,
\quad
p_+(x):=(1+\tanh(\pi x))/2.
\end{equation}
Then $\|p_-\|_V=\|p_+\|_V=2$.

\item[{\rm(b)}]
For every $h\in\R$, put $p_\pm^h(x):=p_\pm(x+h)$. Then
\begin{equation}\label{eq:elementary-2}
\big\|p_\pm -p_\pm^h\big\|_V\le 5\pi|h|/2.
\end{equation}

\item[{\rm(c)}]
For every $m>0$,
\begin{equation}\label{eq:elementary-3}
\int_{\R\setminus[-m,m]}|(p_\pm)'(x)|\,dx<e^{-2\pi m}.
\end{equation}
\end{enumerate}
\end{lemma}
\begin{proof}
(a) Since the function $p_+$ (resp. $p_-$) is monotonically increasing
(resp. decreasing), $p_\pm(\mp\infty)=0$ and $p_\pm(\pm\infty)=1$, we have
$\|p_\pm\|_{L^\infty(\R)}=1$ and $V(p_\pm)=|p_\pm(+\infty)-p_\pm(-\infty)|=1$.
Thus $\|p_\pm\|_V=\|p_\pm\|_{L^\infty(\R)}+V(p_\pm)=2$. Part (a) is proved.

(b) From \eqref{eq:elementary-1} it follows that
\begin{equation}\label{eq:elementary-4}
(p_\pm)'(x)= \pm\frac{\pi}{2\cosh^2(\pi x)},
\quad
(p_\mp)''(x)=\mp\frac{\pi^2\tanh(\pi x)}{\cosh^2(\pi x)},
\quad x\in\R.
\end{equation}
Hence $|(p_\pm)'(x)|\le\pi/2$ for all $x\in\R$. From here, by the mean value theorem, we obtain
\[
|p_\pm(\pi x)-p_\pm[\pi(x+h)]|\le \pi|h|/2,\quad x,h\in\R,
\]
whence
\begin{equation}\label{eq:elementary-5}
\|p_\pm-p_\pm^h\|_{L^\infty(\R)}\le\pi|h|/2.
\end{equation}
Taking into account identities \eqref{eq:elementary-4}, we obtain
\[
|p_\pm''(x)|\le 2\pi p_+'(x),\quad x\in\R.
\]
Then for $h\in\R$,
\begin{align}
V\big(p_\pm-p_\pm^h\big)
&=
\int_\R|p_\pm'(x)-p_\pm'(x+h)|\,dx
=
\int_\R \left|\int_x^{x+h}p_\pm''(y)dy\right|\,dx
\notag\\
&\le
\int_\R\,dx \int_x^{x+|h|}|p_\pm''(y)|\,dy
\le
2\pi\int_\R\,dx \int_x^{x+|h|}p_+'(y)\,dy
\notag\\
&=
2\pi \int_\R p_+'(y)\,dy \int^y_{y-|h|}\,dx
=
2\pi|h|\big(p_+(+\infty)-p_+(-\infty)\big)=2\pi|h|.
\label{eq:elementary-6}
\end{align}
Combining \eqref{eq:elementary-5} and \eqref{eq:elementary-6}, we arrive at \eqref{eq:elementary-2}.

(c) From \eqref{eq:elementary-1} it follows that for $m>0$,
\[
\int_{\R\setminus[-m,m]}|p_\pm'(x)|\,dx
=\pi\int_m^{+\infty}\frac{dx}{\cosh^2(\pi x)}
=1-\tanh(\pi m)
=\frac{1}{e^{2\pi m}+1}<e^{-2\pi m},
\]
which completes the proof.
\end{proof}
\subsection{Limiting Values of Elements of \boldmath{$\widetilde{\cE}(\R_+,V(\R))$}}
For functions in the algebra $\fa\in\widetilde{\cE}(\R_+,V(\R))$, we have a
stronger result than Lemma~\ref{le:values}, which follows from \cite[Lemma 2.9]{K06}
with the aid of the diagonal process.
\begin{lemma}\label{le:values-tilde}
Let $s\in\{0,\infty\}$ and $\{\fa_k\}_{k=1}^\infty$ be a countable subset of the
algebra $\widetilde{\cE}(\R_+,V(\R))$. For each $\xi\in M_s(SO(\R_+))$ there
is a sequence $\{t_j\}_{j\in\N}\subset\R_+$ and functions
$\fa_k(\xi,\cdot)\in V(\R)$ such that $t_j\to s$ as $j\to\infty$
and
\begin{equation}\label{eq:values-tilde}
\lim_{j\to\infty}\|\fa_k(t_j,\cdot)-\fa_k(\xi,\cdot)\|_V=0
\quad\mbox{for all}\quad k\in\N.
\end{equation}
Conversely, every sequence $\{\tau_j\}_{j\in\N}
\subset\R_+$ such that $\tau_j\to s$ as $j\to\infty$
contains a subsequence $\{t_j\}_{j\in\N}$ such that
\eqref{eq:values-tilde} holds for some $\xi\in M_s(SO(\R_+))$.
\end{lemma}
As usual, the maximal ideal space $M(SO(\R_+))$ is equipped with the Gelfand topology.
Then, in view of \cite[Section~1.24]{BS06}, the set $\Delta$ is a compact Haudorff
subspace of $M(SO(\R_+))$. It is equipped with the induced topology. Finally,
the compact Hausdorff space $\Delta\times\overline{\R}$ is equipped with the
product topology generated by the topologies of $\Delta$ and $\overline{\R}$.
\begin{lemma}\label{le:continuity}
For every $\fa\in\widetilde{\cE}(\R_+,V(\R))$, the function $(\xi,x)\mapsto\fa(\xi,x)$
is continuous on the compact Hausdorff space $\Delta\times\overline{\R}$.
\end{lemma}
\begin{proof}
Fix $\eps>0$. It follows from \eqref{eq:definition-cE} that
there exists a $\delta>0$ such that for all $h\in(-\delta,\delta)$,
\[
\sup_{t\in\R_+}\sup_{x\in\R}|\fa(t,x)-\fa(t,x+h)|\le
\sup_{t\in\R_+}\|\fa(t,\cdot)-\fa(t,\cdot+h)\|_V<\eps/6.
\]
Hence there is an $h\in(0,\infty)$ such that, for all $t\in\R_+$ and all $x,y\in\R$ with $|x-y|<h$,
\begin{equation}\label{eq:continuity-1}
|\fa(t,x)-\fa(t,y)|<\eps/6.
\end{equation}
By Lemma~\ref{le:values-tilde}, for every $s\in\{0,\infty\}$ and
$\xi\in M_s(SO(\R_+))$, there is a sequence $\{t_j\}_{j\in\N}$ and a function
$\fa(\xi,\cdot)\in V(\R)\subset C(\overline{\R})$ such that $t_j\to s$ as
$j\to\infty$ and
\begin{equation}\label{eq:continuity-2}
\lim_{j\to\infty}\sup_{x\in\overline{\R}}|\fa(t_j,x)-\fa(\xi,x)|
\le
\lim_{j\to\infty}\|\fa(t_j,\cdot)-\fa(\xi,\cdot)\|_V=0.
\end{equation}
From the above inequality it follows that there is a $J\in\N$ such that for all $j\ge J$,
\[
|\fa(t_j,x)-\fa(\xi,x)|<\eps/6,
\quad
|\fa(t_j,y)-\fa(\xi,x)|<\eps/6.
\]
Combining these inequalities with \eqref{eq:continuity-1}, we deduce for all $x,y\in\R$
satisfying $|x-y|<h$, all $j\ge J$, all $s\in\{0,\infty\}$, and all $\xi\in M_s(SO(\R_+))$ that
\[
|\fa(\xi,x)-\fa(\xi,y)|\le |\fa(t_j,x)-\fa(\xi,x)|+|\fa(t_j,y)-\fa(\xi,y)|+|\fa(t_j,x)-\fa(t_j,y)|<\eps/2.
\]
Therefore, for all $x,y\in\R$ satisfying $|x-y|<h$ we have
\begin{equation}\label{eq:continuity-3}
\sup_{\xi\in\Delta}|\fa(\xi,x)-\fa(\xi,y)|\le\eps/2.
\end{equation}
Fix $\xi\in\Delta$.
Since the function $\fa(\cdot,x)$ belongs to the algebra $SO(\R_+)$, there exists
an open neighborhood $U_x(\xi)\subset\Delta$ of $\xi$ such that
\begin{equation}\label{eq:continuity-4}
|\fa(\eta,x)-\fa(\xi,x)|<\eps/2\quad\mbox{for all}
\quad \eta\in U_x(\xi).
\end{equation}
Consequently, we infer from \eqref{eq:continuity-3} and \eqref{eq:continuity-4} that
\[
|\fa(\eta,y)-\fa(\xi,x)|\le|\fa(\eta,y)-\fa(\eta,x)|
+|\fa(\eta,x)-\fa(\xi,x)|<\eps
\]
for all $(\eta,y)\in U_x(\xi)\times(x-h,x+h)$, which means that the
function $(\xi,x)\mapsto\fa(\xi,x)$ is continuous on $\Delta\times\R$.

It remains to show that actually the function $(\xi,x)\mapsto\fa(\xi,x)$
is continuous on $\Delta\times\overline{\R}$.
By \eqref{eq:definition-tilde-cE}, for every $\eps>0$ there is an $M>0$ such that
\begin{equation}\label{eq:continuity-5}
\sup_{t\in\R_+}|\fa(t,y)-\fa(t,+\infty)|\le\sup_{t\in\R_+}\int_M^\infty
|\partial_x\fa(t,x)|\,dx<\eps/6 \quad\mbox{for all}\quad y>M.
\end{equation}
By Lemma~\ref{le:values-tilde}, for every $s\in\{0,\infty\}$ and every $\xi\in M_s(SO(\R_+))$
there exist a sequence $\{t_j\}_{j\in\N}$ and a function $\fa(\xi,\cdot)\in V(\R)\subset C(\overline{\R})$
such that $t_j\to s$ as $j\to\infty$ and \eqref{eq:continuity-2} is fulfilled.
From \eqref{eq:continuity-2} it follows that there is a $J\in\N$ such that for
all $j\ge J$, all $s\in\{0,\infty\}$, and all $\xi\in M_s(SO(\R_+))$,
\[
|\fa(\xi,y)-\fa(\xi,+\infty)|
\le
|\fa(t_j,y)-\fa(\xi,y)|+
|\fa(t_j,+\infty)-\fa(\xi,+\infty)|+
|\fa(t_j,y)-\fa(t_j,+\infty)|<\eps/2.
\]
Therefore, for all $y>M$ we have
\begin{equation}\label{eq:continuity-6}
\sup_{\xi\in\Delta}|\fa(\xi,y)-\fa(\xi,+\infty)|\le\eps/2.
\end{equation}
Fix $\xi\in\Delta$. Since the function $\fa(\cdot,+\infty)$ belongs to
$SO(\R_+)$, there is an open neighborhood $U_{+\infty}(\xi)\subset \Delta$
of $\xi$ such that
\begin{equation}\label{eq:continuity-7}
|\fa(\eta,+\infty)-\fa(\xi,+\infty)|<\eps/2 \quad\mbox{for all}
\quad \eta\in U_{+\infty}(\xi).
\end{equation}
Then similarly to \eqref{eq:continuity-4} we deduce from \eqref{eq:continuity-6} and
\eqref{eq:continuity-7} that
\begin{equation}\label{eq:continuity-8}
|\fa(\eta,y)-\fa(\xi,+\infty)|
\le
|\fa(\eta,y)-\fa(\eta,+\infty)|+|\fa(\eta,+\infty)-\fa(\xi,+\infty)|<\eps
\end{equation}
for all $(\eta,y)\in U_{+\infty}(\xi)\times(M,+\infty]$.

Analogously, for every $\xi\in\Delta$ there exist an open neighborhood
$U_{-\infty}(\xi)\subset\Delta$ of $\xi$ and a number $M<0$
such that
\begin{equation}\label{eq:continuity-9}
|\fa(\eta,y)-\fa(\xi,-\infty)|<\eps
\end{equation}
for all $(\eta,y)\in U_{-\infty}(\xi)\times[-\infty,M)$.

Finally, we conclude from \eqref{eq:continuity-8}--\eqref{eq:continuity-9} and
the continuity of $(\xi,x)\mapsto\fa(\xi,x)$ on the set
$\Delta\times\R$ that this function is continuous on the compact
Hausdorff space $\Delta\times\overline{\R}$.
\end{proof}
\subsection{Key Construction}
In this subsection we show that if $\fa\in\widetilde{\cE}(\R_+,V(\R))$
does not degenerate on the ``boundary" \eqref{eq:boundary}, then there exists
$\fb\in\widetilde{\cE}(\R_+,V(\R))$ such that $\fb=1/\fa$ on the ``boundary"
\eqref{eq:boundary}.
\begin{lemma}\label{le:bounded-away}
If $a\in\widetilde{\cE}(\R_+,V(\R))$ and
\begin{equation}\label{eq:bounded-away-1}
\fa(t,\pm\infty)\ne 0
\ \text{ for all }\
t\in\R_+,
\quad
\fa(\xi,x)\ne 0
\ \text{ for all }\
(\xi,x)\in\Delta\times\overline{\R}.
\end{equation}
then
\begin{equation}\label{eq:bounded-away-2}
A_\pm:=\sup_{t\in\R_+}\frac{1}{|a(t,\pm\infty)|}<\infty
\end{equation}
and there exists an $r>1$ such that
\begin{equation}\label{eq:bounded-away-3}
A(r):=\sup_{(t,x)\in T_r\times\overline{\R}}\left|\frac{1}{\fa(t,x)}\right|<\infty
\end{equation}
where $T_r:=(0,r^{-1}]\cup[r,\infty)$.
\end{lemma}
\begin{proof}
By Lemma~\ref{le:continuity}, the function $(\xi,x)\mapsto\fa(\xi,x)$
is continuous on the compact Hausdorff space $\Delta\times\overline{\R}$.
Therefore, we infer from \eqref{eq:bounded-away-1} that
\begin{equation}\label{eq:bounded-away-4}
C:=\min\{|\fa(\xi,x)|:(\xi,x)\in\Delta\times\overline{\R}\}>0.
\end{equation}
For every point $(\xi,x)\in\Delta\times\overline{\R}$ we consider its open
neighborhood $U_{\fa,\xi,x}\subset M(SO(\R_+))\times\overline{\R}$
such that
\begin{equation}\label{eq:bounded-away-5}
|\fa(\eta,y)-\fa(\xi,x)|<C/2\quad \mbox{ for every }\quad
(\eta,y)\in U_{\fa,\xi,x}.
\end{equation}
We claim that there exists a number $r>1$ such that
\begin{equation}\label{eq:bounded-away-6}
T_r\times\overline{\R}\subset\bigcup_{(\xi,x)\in\Delta\times\overline{\R}}
U_{\fa,\xi,x}.
\end{equation}
Assume the contrary. Then for every $n\in\N\setminus\{1\}$ there exists a point
$(\tau_n,x_n)\in T_n\times\overline{\R}$ such that
\begin{equation}\label{eq:bounded-away-7}
(\tau_n,x_n)\notin
\left(\bigcup_{(\xi,x)\in M_0(SO(\R_+))\times\overline{\R}}U_{\fa,\xi,x}\right)
\cup
\left(\bigcup_{(\xi,x)\in M_\infty(SO(\R_+))\times\overline{\R}}U_{\fa,\xi,x}\right).
\end{equation}
Since $\tau_n\in T_n=(0,1/n]\cup[n,\infty)$ for all $n\ge 2$, we can extract a subsequence
$\{\tau_{n_k}\}_{k\in\N}$ of the sequence $\{\tau_n\}_{n\in\N\setminus\{1\}}$ such that
\begin{equation}\label{eq:bounded-away-8}
\lim_{k\to\infty}\tau_{n_k}=s
\quad\mbox{for some}\quad
s\in\{0,\infty\}.
\end{equation}
Further, we can extract a subsequence $\left\{x_{n_{k_i}}\right\}_{i\in\N}$ of the
corresponding sequence $\{x_{n_k}\}_{k\in\N}$ such that the limit
\begin{equation}\label{eq:bounded-away-9}
x_0:=\lim_{i\to\infty}x_{n_{k_i}}\in\overline{\R}
\end{equation}
exists. Then, by Lemma~\ref{le:values-tilde}, there exists a subsequence
$\{t_j\}_{j\in\N}=\left\{\tau_{n_{k_{i_j}}}\right\}_{j\in\N}$ of the sequence
$\left\{\tau_{n_{k_i}}\right\}_{i\in\N}$ and a point $\xi_0\in M_s(SO(\R_+))$
such that
\begin{equation}\label{eq:bounded-away-10}
\lim_{j\to\infty}\|\fa(t_j,\cdot)-\fa(\xi_0,\cdot)\|_V=0.
\end{equation}
Put $\{y_j\}_{j\in\N}=\left\{x_{n_{k_{i_j}}}\right\}_{j\in\N}$. Taking into
account \eqref{eq:bounded-away-7}--\eqref{eq:bounded-away-10},
we have shown that if \eqref{eq:bounded-away-6}
is violated for all $r>1$, then there exist $s\in\{0,\infty\}$, $\xi_0\in M_s(SO(\R_+))$,
and a sequence $\{(t_j,y_j)\}_{j\in\N}$ such that \eqref{eq:bounded-away-10}
is fulfilled,
\begin{equation}\label{eq:bounded-away-11}
\{(t_j,y_j): j\in\N\}\cap
\left(\bigcup_{(\xi,x)\in
M_s(SO(\R_+))\times\overline{\R}}U_{\fa,\xi,x}\right)
=\emptyset,
\end{equation}
and
\begin{equation}\label{eq:bounded-away-12}
\lim_{j\to\infty}y_j=x_0\in\overline{\R},
\quad
\lim_{j\to\infty}t_j=s.
\end{equation}
Since $(\xi_0,x_0)\in M_s(SO(\R_+))\times\overline{\R}\subset\Delta\times\overline{\R}$,
from Lemma~\ref{le:continuity} and the first equality in \eqref{eq:bounded-away-12}
we deduce that
\begin{equation}\label{eq:bounded-away-13}
\lim_{j\to\infty}|\fa(\xi_0,y_j)-\fa(\xi_0,x_0)|=0.
\end{equation}
For every $j\in\N$, we have
\begin{align*}
|\fa(t_j,y_j)-\fa(\xi_0,x_0)|
&\le
|\fa(t_j,y_j)-\fa(\xi_0,y_j)|+|\fa(\xi_0,y_j)-\fa(\xi_0,x_0)|
\\
&\le
\sup_{y\in\overline{\R}}|\fa(t_j,y)-\fa(\xi_0,y)|+|\fa(\xi_0,y_j)-\fa(\xi_0,x_0)|
\\
&\le
\|\fa(t_j,\cdot)-\fa(\xi_0,\cdot)\|_V+|\fa(\xi_0,y_j)-\fa(\xi_0,x_0)|.
\end{align*}
From \eqref{eq:bounded-away-10}, \eqref{eq:bounded-away-13}, and the above inequality
we deduce that
\[
\lim_{j\to\infty}\fa(t_j,y_j)=\fa(\xi_0,x_0).
\]
This means that for all sufficiently large $j$ the points $(t_j,y_j)$
belong to the neighborhood $U_{\fa,\xi_0,x_0}$ of the point
$(\xi_0,x_0)\in M_s(SO(\R_+))\times\overline{\R}$, which is impossible in view
of \eqref{eq:bounded-away-11}. Hence, we arrive at the contradiction.

Thus, condition \eqref{eq:bounded-away-6} is fulfilled for some $r>1$.
Therefore, in view of \eqref{eq:bounded-away-4} and \eqref{eq:bounded-away-5},
we obtain
\[
\inf_{(t,x)\in T_r\times\overline{\R}}|\fa(t,x)|>C/2>0.
\]
This inequality immediately yields \eqref{eq:bounded-away-3}. Finally,
\eqref{eq:bounded-away-3} and the first condition in \eqref{eq:bounded-away-1}
imply \eqref{eq:bounded-away-2}.
\end{proof}
\begin{lemma}\label{le:b-construction}
Suppose $a\in\widetilde{\cE}(\R_+,V(\R))$ satisfies \eqref{eq:bounded-away-1}
and $r>1$ is a number such that \eqref{eq:bounded-away-3} holds (the existence of this
number is guaranteed by Lemma~{\rm\ref{le:bounded-away}}). Put
\begin{equation}\label{eq:b-construction-1}
\ell_\pm(t):=\frac{\ln r\pm\ln t}{2\ln r},
\quad
c_\pm(t):=\frac{1}{\fa(t,\pm\infty)}-\frac{\ell_-(t)}{\fa(r^{-1},\pm\infty)}-\frac{\ell_+(t)}{\fa(r,\pm\infty)},
\quad t\in[r^{-1},r],
\end{equation}
and consider the functions $p_\pm$ given by \eqref{eq:elementary-1}.
Then the function
\begin{equation}\label{eq:b-construction-2}
\fb(t,x):=\left\{\begin{array}{l}
\displaystyle\frac{1}{\fa(t,x)}, \quad (t,x)\in(\R_+\setminus[r^{-1},r])\times\overline{\R},
\\[3mm]
\displaystyle\frac{\ell_-(t)}{\fa(r^{-1},x)}+\frac{\ell_+(t)}{\fa(r,x)}+c_-(t)p_-(x)+c_+(t)p_+(x),
\quad (t,x)\in[r^{-1},r]\times\overline{\R},
\end{array}\right.
\end{equation}
is continuous on $\R_+\times\overline{\R}$ and is equal to $1/\fa$ on the
set $((\R_+\setminus(r^{-1},r))\times\overline{\R}\big)\cup\big((r^{-1},r)\times\{\pm\infty\}\big)$.
\end{lemma}
\begin{proof}
Since $\ell_\pm(r^{\mp 1})=0$ and $\ell_\pm(r^{\pm 1})=1$, we have $c_\pm(r)=c_\pm(r^{-1})=0$.
Therefore
\begin{equation}\label{eq:b-construction-3}
\fb(r^{\pm 1},x)=1/\fa(r^{\pm 1},x)
\quad\mbox{for all}\quad
x\in\R.
\end{equation}
Taking into account that $p_\mp(\pm\infty)=0$ and $p_\pm(\pm\infty)=1$, we get
from \eqref{eq:b-construction-1}--\eqref{eq:b-construction-2}
\begin{equation}\label{eq:b-construction-4}
\fb(t,\pm\infty)=
\frac{\ell_-(t)}{\fa(r^{-1},\pm\infty)}+
\frac{\ell_+(t)}{\fa(r,\pm\infty)}+
c_\pm(t)=\frac{1}{\fa(t,\pm\infty)}
\quad\mbox{for all}\quad t\in[r^{-1},r].
\end{equation}
Thus, the assertion of the lemma follows from \eqref{eq:b-construction-3}--\eqref{eq:b-construction-4}
and and the equality $\fb(t,x)=1/\fa(t,x)$ for all $(t,x)\in(\R_+\setminus[r^{-1},r])\times\overline{\R}$
(see \eqref{eq:b-construction-2}).
\end{proof}
\begin{lemma}\label{le:b-symbol}
Suppose $a\in\widetilde{\cE}(\R_+,V(\R))$ satisfies \eqref{eq:bounded-away-1}
and $\fb$ is the function defined by \eqref{eq:b-construction-1}--\eqref{eq:b-construction-2}
with $r>1$ such that \eqref{eq:bounded-away-3} holds (the existence of this
number is guaranteed by Lem\-ma~{\rm\ref{le:bounded-away}}). Then $\fb\in\widetilde{\cE}(\R_+,V(\R))$
and
\begin{equation}\label{eq:b-symbol-0}
\fb(t,\pm\infty)=1/\fa(t,\pm\infty)\
\mbox{ for all }\ t\in\R_+,
\quad
\fb(\xi,x)=1/\fa(\xi,x)\
\mbox{ for all }\ (\xi,x)\in\Delta\times\overline{\R}.
\end{equation}
\end{lemma}
\begin{proof}
We divide the proof into five steps:

(a) First we prove that the function $\fb$ belongs to the algebra $C_b(\R_+,V(\R))$.
Let
\[
T_r:=(0,r^{-1}]\cup[r,+\infty).
\]
By Lemma~\ref{le:b-construction},
\begin{equation}\label{eq:b-symbol-1}
\fb(t,x)=1/\fa(t,x), \quad (t,x)\in T_r\times\overline{\R}.
\end{equation}
Since $\fa(t,\cdot)$ belongs to $V(\R)$ for all $t\in\R_+$, by analogy with \eqref{eq:IC-1},
we infer from \eqref{eq:bounded-away-3} that
\begin{equation}\label{eq:b-symbol-2}
\|\fb(t,\cdot)\|_V\le A^2(r)\sup_{t\in T_r}\|\fa(t,\cdot)\|_V,
\quad t\in T_r.
\end{equation}
From \eqref{eq:bounded-away-2} and \eqref{eq:b-construction-1} it follows that
\begin{equation}\label{eq:b-symbol-3}
0\le\ell_\pm(t)\le 1,
\quad
|c_\pm(t)|\le 3A_\pm,
\quad t\in[r^{-1},r].
\end{equation}
From \eqref{eq:b-construction-2}, \eqref{eq:b-symbol-1}--\eqref{eq:b-symbol-3},
and Lemma~\ref{le:elementary}(a) it follows that for $t\in(r^{-1},r)$,
\begin{align}
\|\fb(t,\cdot)\|_V
&\le
\ell_-(t)\|\fb(r^{-1},\cdot)\|_V+\ell_+(t)\|\fb(r,\cdot)\|_V+|c_-(t)|\,\|p_-\|_V+|c_+(t)|\,\|p_+\|_V
\notag\\
&\le
2A^2(r)\sup_{t\in T_r}\|\fa(t,\cdot)\|_V+6A_-+6A_+.
\label{eq:b-symbol-4}
\end{align}
Combining \eqref{eq:b-symbol-2} and \eqref{eq:b-symbol-4}, we arrive at
\begin{equation}\label{eq:b-symbol-5}
\|\fb(\cdot,\cdot)\|_{C_b(\R_+,V(\R))}=\sup_{t\in\R_+}\|\fb(t,\cdot)\|_V
\le
2A^2(r)\sup_{t\in T_r}\|\fa(t,\cdot)\|_V+6A_-+6A_+<+\infty.
\end{equation}

From \eqref{eq:bounded-away-3} and \eqref{eq:b-symbol-1}--\eqref{eq:b-symbol-2},
by analogy with \eqref{eq:IC-3}, we obtain for $t,\tau\in T_r$,
\begin{align*}
\|\fb(t,\cdot)-\fb(\tau,\cdot)\|_V
&\le
\|\fb(t,\cdot)\|_V\|\fb(\tau,\cdot)\|_V
\|\fa(t,\cdot)-\fa(\tau,\cdot)\|_V
\notag\\
&\le
A^4(r)\left(\sup_{t\in T_r}\|\fa(t,\cdot)\|_V\right)^2
\|\fa(t,\cdot)-\fa(\tau,\cdot)\|_V.
\end{align*}
Since $\fa$ is a continuous $V(\R)$-valued function, from the above inequality
we conclude that $t\mapsto\fb(t,\cdot)$ is a continuous $V(\R)$-valued function for
$t\in T_r$.

Obviously, $\ell_\pm$ are continuous on $[r^{-1},r]$. Since $\fa$ is a continuous
$V(\R)$-valued function, taking into account \eqref{eq:bounded-away-2}, we also have
for $t,\tau\in[r^{-1},r]$,
\[
\left|\frac{1}{\fa(t,\pm\infty)}-\frac{1}{\fa(\tau,\pm\infty)}\right|
=
\frac{|\fa(t,\pm\infty)-\fa(\tau,\pm\infty)|}{|\fa(t,\pm\infty)|\,|\fa(\tau,\pm\infty)|}
\le
A_\pm^2\|\fa(t,\cdot)-\fa(\tau,\cdot)\|_V.
\]
From this inequality and the definitions of $c_\pm$ in \eqref{eq:b-construction-1}
we see that the functions $c_\pm$ are continuous on $[r^{-1},r]$. Therefore,
from the definition \eqref{eq:b-construction-2} we conclude that $t\mapsto\fb(t,\cdot)$
is a continuous $V(\R)$-valued function on $[r^{-1},r]$. From the continuity of the
$V(\R)$-valued function $t\mapsto\fb(t,\cdot)$ on $\R_+$ and inequality
\eqref{eq:b-symbol-5} we conclude that $\fb\in C_b(\R_+,V(\R))$.

\medskip
(b) Now we prove that $\fb\in SO(\R_+,V(\R))$. By analogy with \eqref{eq:IC-4},
from \eqref{eq:bounded-away-3} and \eqref{eq:b-symbol-1} we obtain
\[
\|\fb(t,\cdot)-\fb(\tau,\cdot)\|_{L^\infty(\R)}
\le
A^2(r)\|\fa(t,\cdot)-\fa(\tau,\cdot)\|_{L^\infty(\R)},
\quad t,\tau\in T_r.
\]
Since $a\in SO(\R_+,V(\R))$, from this estimate we obtain
\[
\lim_{\nu\to s}\operatorname{cm}_\nu^C(\fb)\le A^2(r)\lim_{\nu\to s}\operatorname{cm}_\nu^C(\fa)=0,
\]
which means that $\fb\in SO(\R_+,V(\R))$.

\medskip
(c) On this step we show that $\fb\in\cE(\R_+,V(\R))$. By analogy with \eqref{eq:IC-5},
taking into account that the norm of $V(\R)$ is translation-invariant, from
\eqref{eq:bounded-away-3} and \eqref{eq:b-symbol-1}--\eqref{eq:b-symbol-2}
we get for $h\in\R$ and $t\in T_r$,
\begin{align}
\|\fb(t,\cdot)-\fb^h(t,\cdot)\|_V
&\le
\|\fb(t,\cdot)\|_V\|\fb^h(t,\cdot)\|_V\|\fa(t,\cdot)-\fa^h(t,\cdot)\|_V
\notag\\
&\le
C(\fa)\sup_{t\in\R_+}\|\fa(t,\cdot)-\fa^h(t,\cdot)\|_V,
\label{eq:b-symbol-6}
\end{align}
where
\[
C(\fa):=A^4(r)\left(\sup_{t\in T_r}\|\fa(t,\cdot)\|_V\right)^2.
\]
On the other hand, from \eqref{eq:b-construction-2}, \eqref{eq:b-symbol-1},
\eqref{eq:b-symbol-3}, \eqref{eq:b-symbol-6}, and Lemma~\ref{le:elementary}(b)
it follows that for $h\in\R$ and $t\in(r^{-1},r)$,
\begin{align}
\|\fb(t,\cdot)-\fb^h(t,\cdot)\|_V
\le &
\ell_-(t)\|\fb(r^{-1},\cdot)-\fb^h(r^{-1},\cdot)\|_V
+
\ell_+(t)\|\fb(r,\cdot)-\fb^h(r,\cdot)\|_V
\notag\\
&+
|c_-(t)|\,\|p_--p_-^h\|_V+|c_+(t)|\,\|p_+-p_+^h\|_V
\notag\\
\le &
2C(\fa)\sup_{t\in\R_+}\|\fa(t,\cdot)-\fa^h(t,\cdot)\|_V
+
\frac{15\pi}{2}(A_-+A_+)|h|.
\label{eq:b-symbol-7}
\end{align}
Combining \eqref{eq:b-symbol-6}--\eqref{eq:b-symbol-7}, we arrive at
\[
\sup_{t\in\R_+}\|\fb(t,\cdot)-\fb^h(t,\cdot)\|_V
\le
2C(\fa)\sup_{t\in\R_+}\|\fa(t,\cdot)-\fa^h(t,\cdot)\|_V
+
\frac{15\pi}{2}(A_-+A_+)|h|.
\]
Since $\fa\in\cE(\R_+,V(\R))$, the right-hand side of the above inequality tends
to zero as $|h|\to 0$. Hence
\[
\lim_{|h|\to 0}\sup_{t\in\R_+}\|\fb(t,\cdot)-\fb^h(t,\cdot)\|_V=0.
\]
Thus, $b\in\cE(\R_+,V(\R))$.

\medskip
(d) Now we prove that $\fb\in\widetilde{\cE}(\R_+,V(\R))$. From \eqref{eq:b-symbol-1}
we obtain
\[
\partial_x\fb(t,x)=-\fa^{-2}(t,x)\partial_x\fa(t,x),
\quad
(t,x)\in T_r\times\R.
\]
From this identity and \eqref{eq:bounded-away-3} it follows that for all
$m>0$ and $t\in T_r$,
\begin{equation}\label{eq:b-symbol-8}
\int_{\R\setminus[-m,m]}|\partial_x\fb(t,x)|\,dx
\le
A^2(r)\sup_{t\in\R_+}\int_{\R\setminus[-m,m]}|\partial_x\fa(t,x)|\,dx.
\end{equation}
On the other hand, from \eqref{eq:b-symbol-1}, \eqref{eq:b-symbol-3},
\eqref{eq:b-symbol-8}, and Lemma~\ref{le:elementary}(c) it follows that
for all $t\in(r^{-1},r)$ and $m>0$,
\begin{align}
\int_{\R\setminus[-m,m]}|\partial_x\fb(t,x)|\,dx
\le&
\ell_-(t)\int_{\R\setminus[-m,m]}|\partial_x\fb(r^{-1},x)|\,dx
+
\ell_+(t)\int_{\R\setminus[-m,m]}|\partial_x\fb(r,x)|\,dx
\notag\\
&+
|c_-(t)|\int_{\R\setminus[-m,m]}|p_-'(x)|\,dx
+
|c_+(t)|\int_{\R\setminus[-m,m]}|p_+'(x)|\,dx
\notag\\
\le&
2A^2(r)\sup_{t\in\R_+}\int_{\R\setminus[-m,m]}|\partial_x\fa(t,x)|\,dx
+3(A_-+A_+)e^{-2\pi m}.
\label{eq:b-symbol-9}
\end{align}
Combining \eqref{eq:b-symbol-8}--\eqref{eq:b-symbol-9}, we obtain for $m>0$,
\[
\sup_{t\in\R_+}\int_{\R\setminus[-m,m]}|\partial_x\fb(t,x)|\,dx
\le
2A^2(r)\sup_{t\in\R_+}\int_{\R\setminus[-m,m]}|\partial_x\fa(t,x)|\,dx
+3(A_-+A_+)e^{-2\pi m}.
\]
Since $\fa\in\widetilde{\cE}(\R_+,V(\R))$, the right-hand side of the
above inequality tends to zero as $m\to\infty$. This implies that
\[
\lim_{m\to\infty}\sup_{t\in\R_+}\int_{\R\setminus[-m,m]}|\partial_x\fb(t,x)|\,dx=0.
\]
Thus, $\fb\in\widetilde{\cE}(\R_+,V(\R))$.

\medskip
(e) Finally, we prove \eqref{eq:b-symbol-0}. The first equality in \eqref{eq:b-symbol-0}
was proved in Lemma~\ref{le:b-construction}. Fix $s\in\{0,\infty\}$.
Since $\fa,\fb\in\widetilde{\cE}(\R_+,V(\R))$, from Lemma~\ref{le:values} it follows
that for each $\xi\in M_s(SO(\R_+))\subset\Delta$ there exists a sequence $\{t_j\}_{j\in\N}\subset\R_+$
and functions $\fa(\xi,\cdot),\fb(\xi,\cdot)\in V(\R)$ such that $t_j\to s$ as $j\to\infty$ and
\begin{equation}\label{eq:b-symbol-10}
\fa(\xi,x)=\lim_{j\to\infty}\fa(t_j,x),
\quad
\fb(\xi,x)=\lim_{j\to\infty}\fb(t_j,x),
\quad
x\in\overline{\R}.
\end{equation}
For all sufficiently large $j$, one has $t_j\in T_r$. Then from \eqref{eq:b-symbol-1}
we get $\fb(t_j,x)=1/\fa(t_j,x)$ for all sufficiently large $j$ and all $x\in\overline{\R}$.
From this equality and \eqref{eq:b-symbol-10} we obtain the second equality in \eqref{eq:b-symbol-0}.
\end{proof}
\subsection{Regularization of Mellin PDO's with Symbols in \boldmath{$\widetilde{\cE}(\R_+,V(\R))$}}
From \cite[Theorem~4.1]{K09} we can extract the following.
\begin{lemma}\label{le:compactness-PDO}
If $\fc\in\widetilde{\cE}(\R_+,V(\R))$, then $\Op(\fc)\in\cK(L^p(\R_+,d\mu))$
if and only if
\begin{equation}\label{eq:compactness-PDO}
\fc(t,\pm\infty)=0
\ \text{ for all }\
t\in\R_+,
\quad
\fc(\xi,x)=0
\ \text{ for all }\
(\xi,x)\in\Delta\times\overline{\R}.
\end{equation}
\end{lemma}
Now we are in a position to prove the main result of the paper.
\begin{theorem}\label{th:Fredholmness-PDO}
Suppose $\fa\in\widetilde{\cE}(\R_+,V(\R))$.
\begin{enumerate}
\item[{\rm(a)}]
If the Mellin pseudodifferential operator $\Op(\fa)$ is Fredholm on the space
$L^p(\R_+,d\mu)$, then
\begin{equation}\label{eq:Fredholmness-PDO-1}
\fa(t,\pm\infty)\ne 0
\ \text{ for all }\
t\in\R_+,
\quad
\fa(\xi,x)\ne 0
\ \text{ for all }\
(\xi,x)\in\Delta\times\overline{\R}.
\end{equation}

\item[{\rm(b)}]
If \eqref{eq:Fredholmness-PDO-1} holds, then the Mellin pseudodifferential
operator $\Op(\fa)$ is Fredholm on the space $L^p(\R_+,d\mu)$ and each its
regularizer has the form $\Op(\fb)+K$, where $K$ is a compact operator
on the space $L^p(\R_+,d\mu)$ and $\fb\in\widetilde{\cE}(\R_+,V(\R))$
is such that
\begin{equation}\label{eq:Fredholmness-PDO-2}
\fb(t,\pm\infty)=1/\fa(t,\pm\infty)
\mbox{ for all }\ t\in\R_+,
\
\fb(\xi,x)=1/\fa(\xi,x)
\mbox{ for all }\ (\xi,x)\in\Delta\times\overline{\R}.
\end{equation}
\end{enumerate}
\end{theorem}
\begin{proof}
Part (a) follows from the necessity portion of \cite[Theorem~4.3]{K09}, which was
obtained on the base of \cite[Theorem~12.2]{K06} and \eqref{eq:def-E},
\eqref{eq:translation-PDO}--\eqref{eq:translation-symbols}.

The proof of part (b) is analogous to the proof of the sufficiency portion
of \cite[Theorem~12.2]{K06}.  If \eqref{eq:Fredholmness-PDO-1} holds, then
by Lemma~\ref{le:b-symbol} there exists a function $\fb\in\widetilde{\cE}(\R_+,V(\R))$
such that \eqref{eq:Fredholmness-PDO-2} is fulfilled. Therefore, the function
$\fc:=\fa\fb-1$ belongs to $\widetilde{\cE}(\R_+,V(\R))$ and
\eqref{eq:compactness-PDO} holds. By Lemma~\ref{le:compactness-PDO}, the operator
$\Op(\fc)=\Op(\fa\fb)-I$ is compact on $L^p(\R_+,d\mu)$. From this observation
and Theorem~\ref{th:comp-semi-commutators-PDO} we obtain
\[
\Op(\fa)\Op(\fb)\simeq\Op(\fa\fb)\simeq I,
\quad
\Op(\fb)\Op(\fa)\simeq\Op(\fa\fb)\simeq I.
\]
Thus, the operator $\Op(\fa)$ is Fredholm and each its regularizer is of the form
$\Op(\fb)+K$, where $K\in\cK(L^p(\R_+,d\mu))$.
\end{proof}
For a symbol $\fa\in C^\infty(\R_+\times\R)$ satisfying \eqref{eq:Hoermander}--\eqref{eq:Grushin}
the corresponding result was obtained in \cite[Theorem~2.6]{R98}.
\subsection*{Acknowledgments}
This work was partially supported by the Funda\c{c}\~ao para a Ci\^encia e a Tecnologia
(Portuguese Foundation for Science and Technology) through the projects PEst-OE/MAT/
\linebreak
UI0297/2014
(Centro de Matem\'atica e Aplica\c{c}\~oes) and PEst-OE/MAT/UI4032/2014
(Centro de An\'alise Funcional e Aplica\c{c}\~oes). The second author was also
supported by the CONACYT Project No. 168104 (M\'exico) and by PROMEP (M\'exico)
via ``Proyecto de Redes".


\label{lastpage-01}
\end{document}